\documentclass{amsart}
\usepackage{color}
\usepackage{amssymb}
\usepackage{amsmath}
\def\indiq{{\bf 1}}

\def\P{{\mathbb P}}
\def\R{{\mathbb R}}

\def\Z{{\mathbb Z}}
\def\N{{\mathbb N}}

\def\cE{{\mathcal E}}

\def\F{{\mathcal F}}
\def\cV{{\mathcal V}}
\def\cG{{\mathcal G}}

\def\vip{\vskip0.2cm}

\newtheorem{theo}{Theorem}
\newtheorem{prop}[theo]{Proposition}

\newtheorem{rem}[theo]{Remark}
\newtheorem{lem}[theo]{Lemma}
\newtheorem{defin}[theo]{Definition}

\newtheorem{ex}[theo]{Example}
\newtheorem{ass}[theo]{Assumption}

\begin{document}

\title{Hawkes processes with variable length memory and an infinite number of components}
\date{September 17, 2015}
\author{Pierre Hodara}
\author{Eva L\"ocherbach}

\address{P. Hodara:  CNRS UMR 8088, D\'epartement de Math\'ematiques, Universit\'e de Cergy-Pontoise,
2 avenue Adolphe Chauvin, 95302 Cergy-Pontoise Cedex, France.}

\email{pierre.hodara@sfr.fr}

\address{E. L\"ocherbach: CNRS UMR 8088, D\'epartement de Math\'ematiques, Universit\'e de Cergy-Pontoise,
2 avenue Adolphe Chauvin, 95302 Cergy-Pontoise Cedex, France.}

\email{eva.loecherbach@u-cergy.fr}

\subjclass[2010]{}


\begin{abstract}
In this paper, we build a model for biological neural nets where the activity of the network is described by nonlinear Hawkes processes having a variable length memory. The particularity of this paper is to deal with an infinite number of components. We propose a graphical construction of the process and we build, by means of a perfect simulation algorithm, a stationary version of the process. To carry out this algorithm, we make use of a Kalikow-type decomposition technique.

Two models are described in this paper. In the first model, we associate to each edge of the interaction graph a saturation threshold that controls the influence of a neuron on another. 
In the second model, we impose a structure on the interaction graph leading to a cascade of spike trains. Such structures, where neurons are divided into layers can be found in retina.

\end{abstract}
\maketitle 

\section{Introduction and main results}
\subsection{Motivation.} This paper's aim is to give a model in continuous time for an infinite system of interacting neurons. Each neuron is represented by its spike train, i.e.\ the series of events of spiking over time. Since neural activity is continuously recorded in time, a time continuous description is natural. The model considered in this paper is an extension to the continuous time framework of a model which has been recently introduced by Galves and L\"ocherbach \cite{ae} in discrete time. 

We consider a countable set of neurons $I.$ The activity of each neuron $i \in I$ is described by a counting process $ Z^i $ where for any $ - \infty < s < t < \infty ,$ $Z^i (]s, t ]) $ records the number of spikes of neuron $i$ during the interval $ ]s, t ]. $ Under suitable assumptions, the sequence of counting processes $ ( Z^i , i \in I) $ is characterized by its intensity process $ (\lambda_t^i , i \in I) $ which is defined through the relation 
$$ \P ( Z^i \mbox{ has a jump in ]t , t + dt ]} | \F_t ) = \lambda_t^i dt , i \in I.$$
Here, $ \F_t $ is the sigma-field generated by $ Z^i ( ] s, u ] ) , s \le u \le t , i \in I.$ 

Our main motivation is to model neural nets. This naturally leads to the following choice of intensity processes:
\begin{equation}\label{eq:intensity1}
\lambda_t^i = \psi_i \left( \sum_{j \in I} h_{ j \to i } \left( \int_{[L_t^i , t [} g_j ( t-s) d Z^j_s \right) \right) ,
\end{equation}
where $\psi_i : \R \to \R_+ $ is the {\it spiking rate function}, $\{ h_{ j \to i } : \R_+ \to  \R , i , j \in I \} $ a family of {\it synaptic weight functions} modeling the influence of neuron $j $ on neuron $i, $ $ g_j : \R_+ \to \R_+ $ a non-increasing {\it decay function}, and 
\begin{equation}\label{eq:lti}
L_t^i = \sup \{ s < t : Z^i ([s]) > 0 \} 
\end{equation} 
the last spiking time before time $t$ of neuron $i$ (with the convention $[s]:=[s,s]$).

The form (\ref{eq:intensity1}) of our intensity process is close to the typical form of the intensity of a multivariate nonlinear Hawkes process.  The original papers of  Hawkes \cite{Hawkes} and Hawkes and Oakes \cite{ho}, introducing the model,  deal with linear intensity functions.  Extensions to the nonlinear case have been considered by Br\'emaud and Massouli\'e \cite{bm}, see also Massouli\'e \cite{mas}, who propose a study of the stability properties of multivariate nonlinear Hawkes processes. Hawkes processes have shown to be important in various fields of applications. To cite just a few, Hansen, Reynaud-Bouret and Rivoirard \cite{hrbr} is an excellent reference proving the use of Hawkes processes as models of spike trains in neuroscience.  Reynaud-Bouret and Schbath \cite{rbs} deal with an application to genome analysis. In a completely different context, Jaisson and Rosenbaum \cite{jr} obtain limit theorems for nearly unstable Hawkes processes in view of applications in financial price modeling. For a general introduction to Hawkes processes and their basic properties we refer the reader to Daley and Vere-Jones \cite{dvj}.

Our form of the intensity (\ref{eq:intensity1})  differs from the classical Hawkes setting by its {\it variable memory structure} introduced through the term $ L_t^i .$ Hence the spiking intensity of a neuron only depends on its history up to its last spike time which is a biologically very plausible assumption on the memory structure of the process. Therefore, our model can be seen as a nonlinear multivariate Hawkes process where the number of components is infinite with a variable memory structure. The interactions between neurons are encoded through the synaptic weight functions $ h_{ j \to i } $ that we are going to specify below. 

\subsection{The setting.}\label{subsec:setting}
We work on a filtered measurable space $ ( \Omega, {\mathcal A}, \mathbb F)$ which we define as follows.  We write $\mathbb M$ for the canonical path space of simple point processes given by
\begin{multline*}
\mathbb M:= \{ {\tt m} = (t_n)_{n \in \Z}: t_0 \leq 0<t_1, t_n \leq t_{n+1},t_n<t_{n+1} \,\,\, \mbox{if} \,\,\, t_n< + \infty \,\,\, \mbox{or} \,\,\,  t_{n+1}>-\infty ; 
\\ \lim_{n \to +\infty} t_n= + \infty, \lim_{n \to -\infty} t_n= - \infty\} .
\end{multline*}
For any ${\tt m} \in \mathbb M,$ any $n \in \Z,$ let $T_n({\tt m})=t_n .$ We identify $ {\tt m} \in \mathbb M$ with the associated point measure $ \mu  = \sum_n \delta_{T_n ({\tt m })} $ and put ${\mathcal M}_t := \sigma \{ \mu (A) : A \in {\mathcal B} (\R), A \subset ] -\infty,t] \},$ $ {\mathcal M} = {\mathcal M}_\infty.$ We will also systematically identify ${\tt m } $ with the associated counting process $ \alpha ( {\tt m }), $ defined by $\alpha_0 ( {\tt m }) = 0 , $ 
$$ \alpha_t ( {\tt m }) = \mu ( ] 0, t ])  \mbox{ if } t \geq 0,  \alpha_t ( {\tt m }) = - \mu ( ]t, 0 ]) , \mbox{ if } t \le 0. $$
Finally we put $( \Omega, {\mathcal A}, \mathbb F):= ( \mathbb M, {\mathcal M}, ( {\mathcal M}_t )_{t \in \R} )^I .$ We write $ (Z^i , i \in I) $ for the canonical multivariate point measure defined on $ \Omega .$

We specify the following parameters: a family of  firing rate functions $ \psi = \{ \psi_i : \R \to \R_+ , i \in I\},$ a family of synaptic weight functions $h = \{ h_{ j \to i } : \R \to \R , i , j \in I \} ,$ a family of functions  $g= \{ g_j : \R_+ \to \R_+  , j \in I\}, $ which are non-increasing. Recall the definition of the last spiking time of neuron $i$ before time $t,$ given in (\ref{eq:lti}). 

\begin{defin}\label{def:1}
A Hawkes process with variable length memory and an infinite number of interacting components with parameters $ ( \psi, h, g) $  is a probability measure $\P $ on $ ( \Omega, {\mathcal A}, \mathbb F)$ such that 
\begin{enumerate}
\item
$\P-$almost surely, for all $i \neq j,  Z^i $ and $ Z^j  $ never jump simultaneously,
\item
for all $i \in I,$ the compensator of $Z^i $ is given by $\nu^i (dt ) = \lambda^i_t dt ,$ where
$$
 \lambda^i_t  =  \psi_i \left( \sum_{j \in I}  h_{j \to i } \left( \int_{[L_t^i , t [} g_j ( t-s) d Z^j_s \right) \right)  .
$$
\end{enumerate}
\end{defin}

\subsection{Main results.} In the case where $I$ is a finite set, under suitable assumptions on the parameters of the process, the existence and construction of $ (Z^i , i \in I)$ is standard (see \cite{bm} and \cite{dfh}). In our case, however, the number of interacting components defining the process is infinite. In such a framework, Delattre, Fournier and Hoffmann \cite{dfh} prove pathwise existence and uniqueness of the processes, however without giving an explicit construction of the process. In the present paper, we show that -- under suitable assumptions  -- a {\it graphical construction} is possible. This graphical construction does not only imply the existence but also the possibility of a perfect simulation of a stationary version of the process (i.e.\ a probability measure $\P $ on $ (\Omega, {\mathcal A}, \mathbb F),$ such that under $\P, $ for all $ u \in \R  $ and all $i \in I,$ the processes $ Z^i ( ] u, u+ t ] ) , t \in \R $ are stationary). These results are achieved via a {\it Kalikow-type decomposition} in two types of models. The first model is a system containing a {\it saturation threshold} for any directed edge $i \to j $ in the interaction graph defined by the synaptic weights. The second model deals with a {\it cascade of spike trains}.

Kalikow-type decompositions are now largely used in the literature for perfect simulation issues and similar scopes. They have been considered first by Ferrari, Maass, Mart\'{i}nez and Ney \cite{fmmn} and in Comets, Fern\'andez and Ferrari \cite{cff}. This type of technique was then studied in a series of papers for perfect simulation issues. See Galves and L\"ocherbach \cite{ae} for an application in the context of neural biological nets in discrete time. The decomposition that we use in the present paper is a non-trivial extension of the previous considerations to the framework of continuous time neural nets. In the case of our first model, it has to be achieved in a random environment. 

To the best of our knowledge, the perfect simulation algorithm constructed in the present paper is the first result in this direction obtained for Hawkes processes with nonlinear intensity functions. The well known work of M\o ller and Rasmussen \cite{mr} on perfect simulation of Hawkes processes deals with linear intensity functions and exploits very heavily the underling branching structure. The precise form of our perfect simulation algorithm is given in Section 5.2. 

\subsection{Assumptions and notations.}
Throughout this paper we suppose that the firing rate functions $\psi_i : \R \to \R_+$ are non-decreasing and bounded by a real number $\Lambda_i .$ Introducing 
\begin{equation}\label{eq:phi}
\phi_i=\frac{\psi_i}{\Lambda_i},
\end{equation}
we assume 

\begin{ass}\label{ass:1}
The functions $ \phi_i , i \in I, $ are uniformly Lipschitz continuous, i.e.\ there exists a positive constant $ \gamma $ 
such that for all $ x, x' \in \R ,$  $i \in I, $
\begin{equation}\label{eq:Lip}
| \phi_i (x ) - \phi_i  ( x' ) | \le \gamma  |x - x' |   .
\end{equation}

\end{ass}
The interactions between neurons are coded via the synaptic weight functions $ h_{j \to i }.$ For each neuron $i,$ we define
$$ {\mathcal V}_{i \to \cdot} = \{ j \in I, j \neq i  : h_{i \to j } \neq 0 \} ,$$
where $0$ denotes the constant function $0.$ As a consequence, ${\mathcal V}_{i \to \cdot} $ is
the set of all neurons that are directly influenced by neuron $i.$ In the same spirit, we put
\begin{equation}\label{eq:vi}
 {\mathcal V}_{\cdot \to i} = \{ j \in I, j \neq i  : h_{j \to i } \neq 0 \} ,
\end{equation}
which is the set of neurons that have a direct influence on $i.$ These sets may be finite or infinite. 

In the following we introduce the two main types of models that we consider, firstly {\it models with saturation thresholds} and secondly, {\it models with a cascade of spike trains}.

\section{Models with saturation threshold.}\label{sec:2}
\subsection{Models with saturation thresholds.}
We suppose that to each directed edge $ j \to i $ is associated a saturation threshold $K_{ j \to i } > 0 $ representing the maximal number of spikes that the synapse $ j \to i  $ is able to support. We suppose that 
\begin{equation}\label{eq:h}
 h_{ j \to i } (x) = W_{j \to i }  (x \wedge K_{ j \to i } ) ,
 \end{equation}
 where $ W_{ j \to i } \in \R $ is called the {\it synaptic weight} of neuron $j $ on $i.$ Moreover we suppose that $ g_{ j } \equiv 1$ for all $j \in I$ and write for short $ g = {\bf 1}.$  Hence we can rewrite
\begin{equation}\label{eq:intensity2}
 \lambda_t^i = \psi_i \left( \sum_{j \in I} W_{ j \to i} (Z^j ( ]L_t^i , t[ ) \wedge K_{j \to i } )\right) .
\end{equation}
We suppose that
\begin{ass}\label{ass:2}
For all $i \in I,W_{i \to i}=0 $ and  
 \begin{equation}\label{eq:summable}
\sup_{ i \in I} \sum_j |W_{j \to i }|  K_{ j \to i } < \infty .
\end{equation}
\end{ass}

Introduce for any $i \in I $ a non-decreasing sequence $( V_i (k ) )_{ k \geq 0 } $ of finite subsets of $ I $ such that $V_i ( 0 ) = \emptyset ,$ $ V_i ( 1 ) =  \{ i \}  ,$ $ V_i (k ) \subset  V_i ( k+1) ,$ $  V_i (k ) \neq V_i ( k+1)$ if $ V_i (k ) \neq {\mathcal V}_{\cdot \to i}   \cup \{ i \}$ and $\bigcup_k V_i (k) ={\mathcal V}_{\cdot \to i}  \cup \{ i \} $ (recall (\ref{eq:vi})). 
The following theorem states that if any neuron has a sufficiently high spontaneous firing activity, then -- under a stronger summability condition on the interactions than (\ref{eq:summable}) -- a unique stationary version of the Hawkes process with saturation threshold exists.
  
\begin{theo}\label{theo:1hc}
Grant Assumptions \ref{ass:1} and \ref{ass:2} and suppose that for all $i \in I, $ 
\begin{equation}\label{eq:deltahc}
\psi_i \geq \delta_i,
\end{equation} 
for some $ \delta_i > 0.$ Impose moreover the following summability condition. 
\begin{equation}\label{eq:horriblehc}
 \sup_{i \in I} \left( \sum_{k \geq 1} \left[ \left( \sum_{j \in V_i(k)} \frac{\Lambda_j-\delta_j}{\delta_i} \right) \left( \sum_{j \in \partial V_i(k-1)} \left| W_{ j \to i } \right| K_{j \to i} \right) \right] \right) < \frac{1}{\gamma} ,
\end{equation}
where $\gamma $ is given in Assumption \ref{ass:1} and $\partial V_i(k-1):=V_i(k) \setminus  V_i(k-1).$ Then there exists a unique probability measure
$\P $ on $( \Omega, {\mathcal A})$ such that under  $\P ,$ the canonical process $\left( Z^i, i \in I \right) $ on $\Omega$ is a stationary nonlinear Hawkes process with variable length memory and an infinite number of interacting components, with parameters $(\psi,  h,g) $ where $h$ is given by (\ref{eq:h}) and $ g= {\bf 1} .$
\end{theo}

\begin{rem}
\vip
(i) Suppose that we have a stronger summability than (\ref{eq:summable}):
$$
\sup_i \left( \sum_{k \geq 1} |V_i(k)| \sum _{ j \in \partial V_i(k-1)} |W_{j \to i} | K_{j \to i} \right) < \infty .
$$
Then (\ref{eq:horriblehc}) holds for $\gamma $ sufficiently small, if we suppose $\inf_{i \in I} \delta_i > 0$ and $\Lambda_i \leq \Lambda$ for all $ i \in I, $ where $ \Lambda > 0 $ is some fixed constant. 
\vip
(ii)
In the above model, there is a saturation threshold for each directed edge between two neurons.  A different model would be a system where the saturation threshold concerns only the global input received by each neuron. This would amount to consider the following intensity
$$
\lambda_t^i= \psi_i \left( K_i^- \vee \left( \left( \sum_{j \in I} W_{j \to i}  Z^j \left( ] L_t^i , t [ \right) \right) \wedge K_i^+ \right) \right) ,
$$
where $K_i^-$ and $K_i^+$ are global saturation thresholds respectively for inhibition and stimulation. Under obvious changes of condition (\ref{eq:horriblehc}), Theorem \ref{theo:1hc} remains also true in this framework.
\vip
(iii)
In \cite{mas}, the author obtains similar results in a slightly different context dealing with truly infinite memory models. The fact that we consider variable length memory Hawkes processes constitutes of course the main difference with respect to \cite{mas}. As a matter of fact, due to this difference, instead of (\ref{eq:deltahc}) and (\ref{eq:horriblehc}), \cite{mas} requires that $\sum_{i \in I} \sum_{j \in I} W_{j \to i} K_{j \to i} < \infty,$ which is a stronger assumption than (\ref{eq:summable}). Finally, let us stress that our approach gives more than only the existence of a stationary solution. It actually gives a graphical construction of the stationary measure, which is not the case in \cite{mas}. 

\end{rem}

\subsection{A Markovian description in terms of a coupled `house of cards'-process.} As in the case of classical Hawkes processes with exponential loss functions, under suitable assumptions, an alternative description of $(Z^i, i \in I)$ via its intensity processes yields a Markovian description of the process. Throughout this subsection, we impose the summability assumption (\ref{eq:summable}). Moreover, we suppose that
\begin{equation}\label{eq:mi}
 \sup_i \Lambda_i | \cV_{ i \to \cdot } | < \infty 
\end{equation}
and that 
\begin{equation}\label{eq:finite}
T_i := \cV_{\cdot \to i } \cup \cV_{i \to \cdot } \mbox{ are finite sets for all } i \in I.
\end{equation}
$T_i$ is the set of neurons that directly influence neuron $i$ or that are directly influenced by it, i.e. the set of neurons $j$ such that $W_{j \to i} \neq 0 $ or $W_{i \to j} \neq 0.$

It is convenient to adopt a description of a process living on the set of directed edges $ \cE = \{ j \to i ,  j \in {\mathcal V}_{ \cdot \to i }, i \in I\} . $ We write $e = j \to i \in \cE $ for a directed edge and introduce for any such $e$ the process $U_t (e)$ defined by
\begin{equation}
{ U}_t (e) =   Z^j \left( ] L_t^i , t [ \right) , t \in \R .
\end{equation}
With this point of view, the neural network is described by the process $( { U_t (e) } , e\in \cE   )_{t \in \R} ,$ taking values in $S := \N^{\cE} .$ Its dynamic is described by its generator defined by
$$
\cG f (\eta) = \sum_{i \in I}  \psi_i\left( \sum_j W_{ j \to i }(\eta ({ j \to i }) \wedge K_{ j \to i }) \right)   \left[ f \left( \eta + \Delta_i \eta \right) - f (\eta) \right] ,
$$ 
where 
$$ (\Delta_i \eta  ) ({ k \to l })  = \left\{  \begin{array}{ll}
- \eta ({ k \to l })  & \mbox{ if } l = i , k \in {\mathcal V}_{ \cdot \to i } \\
1 & \mbox{ if } k = i , l \in {\mathcal V}_{i \to \cdot} \\
0 & \mbox{ else} 
\end{array} \right\} ,
$$
and where $ f \in {\mathcal D} ( \cG ) = \{ f : ||| f  ||| := \sum_{ e \in \cE } \Delta_f ( e) < \infty  \}$ with $ \Delta_f ( e ) = \sup \{ | f (\eta) - f ( \zeta) | : \eta, \zeta \in S, \eta ({e'}) = \zeta ({e'}) \mbox{ for all } e'  \neq e \} .$ 

\begin{rem}
\vip
(i)
Notice that a spike of neuron $i$ does not only affect all neurons $j \in \cV_{ i \to \cdot } $ which receive an additional potential, but also all neurons $ j \in \cV_{ \cdot \to i } ,$ since all $\eta ( j \to i ) $ are reset to $0$ when a spike of neuron $i$ occurs. It is for this reason that we call the above process a coupled `house of cards'-process. 

\vip
(ii) We could also work with  $\tilde S := \{ \eta \in S : \eta (j \to i) \le K_{j \to i} \, \, \forall e=( j \to i) \in \cE\}$ which is the state space of relevant configurations of the process. This would imply to redefine $U_t(j \to i):=  Z^j \left( ] L_t^i , t [ \right) \wedge K_{j \to i}.$
\end{rem}

By Theorem 3.9 of Chapter 1 of Liggett \cite{lig}, (\ref{eq:summable}) together with (\ref{eq:mi}) and (\ref{eq:finite}) implies that $\cG $ is the generator of a Feller process $ (U_t (e) , e \in \cE )_{t \in \R } $ on $S.$ 

\begin{proof}
Under the above conditions, the generator $\cG $ can be rewritten as 
$$ \cG f (\eta ) = \sum_{i } c_{T_i} ( \eta , d \zeta ) [ f ( \eta_i^\zeta) - f ( \eta ) ] , $$
where 
$$ \eta_i^\zeta ( e ) = \left\{ 
\begin{array}{ll}
\eta ( e) & \mbox{ if } e \notin T_i\\
\zeta (e) & \mbox{ if } e \in T_i
\end{array}
\right\} ,$$
and where 
$$ c_{T_i } ( \eta , d \zeta ) = \psi_i ( \sum_j W_{ j \to i } ( \eta ({ j \to i }) \wedge K_{ j \to i } )) \delta_{ \eta + \Delta_i \eta } ( d \zeta ).$$ 
A straightforward calculation shows that the quantity $c_{T_i } = \sup_{\eta } c_{T_i} ( \eta , S) $ defined by Liggett \cite{lig} in formula (3.3) of Chapter 1 can be upper bounded by
$$ c_{T_i } \le \Lambda_i $$ 
and that 
$$ c_{T_i } ( e ) := \sup \{ \| c_{T_i } ( \eta , \cdot ) - c_{T_i} ( \zeta , \cdot ) \|_{TV} : \eta (e') = \zeta (e') \mbox{ for all } e' \neq e \} ,$$
where $\| \cdot \|_{TV} $ denotes the total variation distance, can be controlled by 
$$  c_{T_i } ( e ) \le \left\{
\begin{array}{ll}
\gamma \Lambda_i | W_e| K_e & \mbox { if } e = k \to i ,\\
\Lambda_i & \mbox{ if } e = i \to l , \\
0 & \mbox{ else}
\end{array}
\right\} .$$
The conditions (\ref{eq:summable}), (\ref{eq:mi}) and (\ref{eq:finite}) thus imply that $M$ defined in formula (3.8) of Chapter 1 of Ligget \cite{lig} can be controlled as
$$ M = \sup_e \sum_{ i : e \in T_i } \sum_{ u \neq e } c_{T_i} ( u ) \le 2 \sup_i \Lambda_i \left( \gamma \sum_k |W_{k \to i } | K_{ k \to i } + |\cV|_{ i \to \cdot } \right) < \infty ,$$
and then Theorem 3.9 of Chapter 1 of Liggett \cite{lig} allows to conclude.  
\end{proof} 

As a consequence, we can reformulate Theorem \ref{theo:1hc} in the following way. 

\begin{theo}\label{theo:markov}
Grant the assumptions of Theorem \ref{theo:1hc} and suppose moreover that (\ref{eq:summable}) together with (\ref{eq:mi}) and (\ref{eq:finite}) are satisfied. Then the process $ (U_t^e , e \in \cE )_{t \in \R } $ is ergodic.  
\end{theo}

\begin{proof}
The above Theorem \ref{theo:markov} follows immediately from Theorem \ref{theo:1hc}.
\end{proof}

\begin{rem}
Let us compare the above result to the $M < \varepsilon -$criterion of Liggett \cite{lig}, Theorem 4.1 of Chapter 1. The quantity $M $ has already been introduced above. Moreover, $\varepsilon $ is defined by
\begin{multline*}
\varepsilon : = \inf_{e \in \cE } \; \; \inf_{ \eta_1 ,\eta_2 \in S : \eta_1 ( e') = \eta_2 ( e') \forall e' \neq e} \; \sum_{ k : e \in T_k } \Big[ c_{T_k}  \left( \eta_1 , \{ \zeta : \zeta(e) = \eta_2(e) \} \right)  \\  + c_{T_k} \left( \eta_2 , \{ \zeta : \zeta(e) = \eta_1(e) \} \right) \Big].
\end{multline*}
The sufficient condition for ergodicity in Theorem 4.1 of Liggett \cite{lig} is $M < \varepsilon.$
Note that in our particular case we have $\varepsilon=0$ which can easily be seen by considering $\eta_1$ and $\eta_2 $ with $\eta_1(e)=1 $ and $\eta_2(e)=3$ for some $e \in \cE.$ Consequently the sufficient condition  $M < \varepsilon$ is not satisfied. However, Theorem \ref{theo:markov} implies the ergodicity of the process without the condition $M < \varepsilon .$
\end{rem}

\section{A cascade of spike trains.}
We now describe the second model that we consider in this paper, a {\it cascade of spike trains}.
We suppose that for each $j \in I,$ the function  $ g_j : \R_+ \to \R_+ $ is measurable and non-increasing such that $\int_0^{+ \infty } g_j(x)dx < + \infty.$ $g_j,$ models a leak function. We suppose moreover that 
\begin{equation}\label{eq:h2}
 h_{ j \to i } (x) = W_{ j \to i } \cdot x , 
\end{equation} 
for a family of {\it synaptic weights} $\{ W_{ j \to i } \in \R \} $ satisfying the summability condition
\begin{equation}\label{eq:summable2}
\sup_{ i \in I} \sum_j | W_{ j \to i } | < \infty .
\end{equation}
Neurons $j$ such that $ W_{ j \to i } > 0 $ are called {\it excitatory} for $i, $ if $ W_{ j \to i } < 0 ,$ then $j$ is called {\it inhibitory} for $i.$ 
Finally, we impose the following condition on the structure of interactions.

\begin{ass}\label{ass:3}
The set $I$ of the neurons is divided into layers $(I_n)_{n \in \Z}$ such that we have the partition $I= \sqcup_{n \in \Z}I_n.$ For each $n \in \Z$  and for each $i \in I_n,$ we suppose that
\begin{equation}\label{eq:structure}
{\mathcal V}_{\cdot \to i} \subset I_{n-1}.
\end{equation}
\end{ass}
Therefore, a neuron only receives information from neurons in the layer just above itself. This assumption does not apply to the brain's structure which is far too complicated. But such a structure can be found in simpler nervous tissues like the retina. The nonlinear Hawkes process that we consider in this section is defined by its intensity given by
\begin{equation}\label{eq:intensity3}
 \lambda^i_t  =  \psi_i \left( \sum_{j \in I}  W_{j \to i }  \int_{[L_t^i , t [} g_j ( t-s) d Z^j_s  \right)  ,
\end{equation}
together with the assumption (\ref{eq:structure}) on the structure of the interactions.

\begin{theo}\label{theo:1}
Grant Assumptions \ref{ass:1} and \ref{ass:3} and suppose moreover that the condition (\ref{eq:summable2}) is satisfied.
If 
\begin{multline}\label{eq:horrible}
 \sup_{i \in I} \left( \sum_{k \geq 1} \left[ \left( k \left( \sum_{j \in  V_i(k)} \Lambda_j \right)+1 \right) \right. \right. 
\times  \\ \left. \left. \left(   
\sum_{j \in V_i(k-1)} | W_{j \to i} | \Lambda_j \int_{k-1}^{k} g_j(s) ds +  \sum_{j \in \partial V_i(k-1)} | W_{j \to i} | \Lambda_j \int_0^{k} g_j(s) ds  \right)  \right] \right) < \frac{1}{\gamma}, 
\end{multline}
where $\gamma $ is given in Assumption \ref{ass:1}, then there exists a unique probability measure
$\P $ on $( \Omega, {\mathcal A})$ such that under  $\P ,$ the canonical process $\left( Z^i , i \in I \right) $ on $\Omega$ is a stationary nonlinear Hawkes process with variable length memory and an infinite number of interacting components, with parameters $(\psi,  h,g) $ where $h$ is given by (\ref{eq:h2}). 
\end{theo}

\begin{rem}
Let us once more compare this result with Theorem 3 of \cite{mas}. Instead of our conditions (\ref{eq:horrible}) and Assumption \ref{ass:3} on the structure of interactions, \cite{mas} requires that 
$$
\sum_{i \in I} \sum_{j \in I} W_{j \to i} \int_0^{+\infty} t g_j(t)dt < \infty.
$$ 
A direct comparison with (\ref{eq:horrible}) is difficult, but (\ref{eq:horrible}) seems to be slightly less restrictive.
\end{rem}

We give two examples in which condition (\ref{eq:horrible}) is satisfied in order to illustrate our result.
For the sake of simplicity, we will assume that for all $j, \Lambda_j = \Lambda$ and $g_j=g.$ We will also choose a simple structure for the network; each layer of neurons is identical and indexed by $\Z$ so that we can take $I= \Z^2$ with indexes $i=(i_1,i_2) \in \Z^2$ where $i_1 $ denotes the `site' of the neuron and $i_2$ the `height' of its layer. 
\begin{ex}
We assume that each neuron is only influenced by two neurons in the upper layer
$$
W_{(j_1,j_2) \to (i_1,i_2)}=W \, \indiq_{j_2=i_2-1} \left( \indiq_{j_1=i_1-1} + \indiq_{j_1=i_1+1} \right) ,
$$
where $W \in \R $ is a constant.  
Then condition (\ref{eq:horrible}) can be rewritten as 
$$
2 \Lambda ( 6 \Lambda +1) \int_0^2 g(s)ds
+ \sum_{k \geq 3} \left[ 2 \Lambda (3 k \Lambda +1) \int_{k-1}^k g(s)ds \right] < \frac{1}{W \gamma} .
$$
The above sum is finite if $\int_{k-1}^k g(s)ds=  {\mathcal O} \left( \frac{1}{k^\alpha} \right) $ with $\alpha > 2.$ In the case of an exponential loss function $g(s) = e^{-as}$ for some $a>0,$ (\ref{eq:horrible}) is satisfied if  $\Lambda , a, W, \gamma$ are such that 
$$
(6 \Lambda + 1) \left( 1- e^{-2a} \right) + e^{-2a} \left( 9 \Lambda + 1 + \frac{3 \Lambda e^{-a}}{1-e^{-a}} \right) < \frac{a}{2\Lambda W \gamma}.
$$
In particular, for fixed $a, \Lambda $ and $\gamma $ there exists $W^* $ such that $W < W^*$ implies (\ref{eq:horrible}).
\end{ex}

\begin{ex}
Suppose that $g(s)= e^{-as}$ for some $a>0.$ We assume that each neuron is influenced by an infinite number of neurons and that
$$
W_{(j_1,j_2) \to (i_1,i_2)}=\frac{W}{ |j_1-i_1|^\beta}  \indiq_{j_2=i_2-1}  \indiq_{j_1 \neq i_1}
$$
for some $ \beta > 0$ and $W > 0.$ Putting $V_{(i_1,i_2)}(k):= \{(j_1,i_2-1); 1 \leq |j_1-i_1| \leq k \},$ we have for all neurons $j \in \partial V_i(k-1), W_{j \to i}= \frac{1}{k^\beta}.$
Condition (\ref{eq:horrible}) can be rewritten as
$$
\sum_{k \geq 2} \left[ \Big( k \Lambda (2k+1) +1 \Big) \left(   
\sum_{l = 1}^{k-1} \frac{\Lambda }{l^\beta} e^{-ak} (e^a-1) + \frac{\Lambda }{k^\beta} (1-e^{-ak} )  \right) 
 \right] < \frac{a}{2 W \gamma} .
$$
This sum is finite if and only if $\beta > 3.$ In particular, for fixed $a, \Lambda , \gamma $ and $\beta > 3 $ there exists $W^* $ such that $W < W^*$ implies (\ref{eq:horrible}).
\end{ex}

\begin{rem}
In the model of Section \ref{sec:2} the presence of spontaneous spikes guarantees that the neural network is always active. In the frame of cascades of spike trains, we do not impose the presence of spontaneous spikes. Thus we have to study the non-extinction of the process. Indeed, if for all $i \in I, \psi_i(0)=0,$ then $Z^i \equiv {\bf 0}$ for all $i \in I$  is a possible stationary version of the Hawkes process with intensity (\ref{eq:intensity3}), and Theorem \ref{theo:1} implies that this is indeed the unique stationary solution.

On the other hand, if there exists $i \in I$ such that $\psi_i(0) > 0,$ then $Z^i \equiv {\bf 0}$ for all $i \in I$ cannot be a stationary solution, i.e.\ the system cannot go extinct.

\end{rem}

\section{A dominating Poisson Random Measure.}
Recall that the firing rate functions $\psi_i $ considered in this paper are bounded by a constant $\Lambda_i$ for each $i \in I.$ We will use this assumption to introduce a Poisson Random Measure (PRM) $N (dt,di,dz) $ on $\R \times I \times [0,1] $ with intensity $ dt  \left( \sum_{i \in I} \Lambda_i \delta_i \right) dz  $ on $\R \times I \times [0,1] $  dominating the process $\left( Z^i, i \in I \right).$ This allows us to select the jump times of $Z^i$ among those of $N^i$ according to probabilities driven by the function $\phi_i$ (recall (\ref{eq:phi})). This leads to the following definition:

\begin{defin}\label{def:2}
A family $ (Z^i , i \in I)$ of random point measures defined on $ ( \Omega, {\mathcal A}, \mathbb F)$ is said to be a Hawkes process with variable length memory and an infinite number of interacting components with parameters $(\psi, h, g ) $ if almost surely, for all $i \in I $ and $C \in {\mathcal B} ( \R ) , $ 
\begin{equation}\label{eq:sde}
Z^i(C)= \int_C \int_{ \{i\} } \int_{[0,1]} 1_{z \leq \frac{1}{\Lambda_i} \psi_i \left( \sum_{ j } h_{j \to i } \left( \int_{[L_t^i , t [} g_j(t-s) d Z^j_s \right) \right) }
{N}(dt,di,dz).
\end{equation}
\end{defin}

According to Br\'emaud and Massouli\'e \cite{bm}, see also Proposition 3 of Delattre, Fournier and Hoffmann \cite{dfh}, a Hawkes process according to Definition \ref{def:2} is a Hawkes process according to Definition \ref{def:1} and vice versa. 

Formula (\ref{eq:sde}) implies that we can construct the process $( Z^i , i \in I) $ by a thinning procedure applied to the a priori family of dominating Poisson random measures $ N^i (dt ) = N ( dt, \{ i \}, [0, 1])$ having intensity $\Lambda_i dt $ each. Since $Z^i$ is a simple point measure, it is enough to define it through the times of its atoms. Each atom of $Z^i $ must also be an atom of $N^i $ since $Z^i \ll N^i.$ In other words, the associated counting process $\alpha ( Z^i)$ can only jump when the counting process $\alpha ({N}^i) $ jumps.
We write $T_n^i, n \in \Z,$ for the jump times of $N^i .$ Fix a given jump $ t = T_n^i $ of $ N^i .$ Then, conditionally on $N,$ the probability that this time is also a jump time of $Z^i$ is given by 
\begin{equation}\label{eq:transition2hc}
 \P( Z^i (\{ t\} ) = 1 | { \mathcal F}_{t}  ) = \phi_i  \left( \sum_j  h_{ j \to i} \left( \int_{[ L_t^i , t[} g_j ( t- s) d Z^j_s  \right)  \right) =: p_{(i,t) } ( 1 | \F _t) .
\end{equation}
In other words, given that $t$ is a jump time of $N^i , $ $p_{(i,t) } ( 1 | \F _t) $ is the probability that this jump is also a jump of $Z^i.$ This probability depends on the past before time $t$ of the process. In what follows we propose a decomposition of $p_{(i,t) } ( 1 | \F _t)$ according to growing time-space neighbourhoods that explore the {\it relevant past} needed in order to determine $p_{(i,t) } ( 1 | \F _t).$ This decomposition is a Kalikow-type decomposition as considered first by Ferrari, Maass, Mart\'{i}nez and Ney \cite{fmmn} and in Comets, Fern\'andez and Ferrari \cite{cff}. The decomposition that we consider here is a nontrivial extension of the previous results to the framework of continuous time neural nets. In the case of Theorem \ref{theo:1}, it has to be achieved in a random environment, where the environment is given by the a priori realization of the PRM $N.$ We start with the proof of Theorem \ref{theo:1hc} which is conceptually simpler.

\section{Proof of Theorem  \ref{theo:1hc} .}
\subsection{Kalikow-type decomposition.} 
The condition (\ref{eq:horriblehc}) of Theorem \ref{theo:1hc} allows us to decompose the law of the conditional probability (\ref{eq:transition2hc}) according to space neighbourhoods $V_i(k).$ 
This decomposition will be independent of the realization of the a priori PRM $N.$ This will be crucial in the perfect simulation procedure described in the next subsection. We will work with $\tilde S ,$  the state space of relevant configurations of the process defined in Remark 6. For the convenience of the reader, we recall its definition here:
$$
\tilde S := \{ \eta \in S : \eta (j \to i) \le K_{j \to i} \, \, \forall e=( j \to i) \in \cE\}
$$
and introduce the following notations. First,
\begin{equation}
r_i^{[0]} ( 1) = \inf_{ \eta  \in \tilde S} \phi_i \left( \sum_j W_{ j \to i } \eta({j \to i} )  \right),
\end{equation}
which is the minimal probability that neuron $i$ spikes uniformly with respect to all possible configurations. Then we define
\begin{equation}
r_i^{[0]} ( 0) = \inf_{ \eta\in \tilde S} \left( 1-\phi_i \left( \sum_j W_{ j \to i } \eta({j \to i} ) \right)  \right),
\end{equation}
which is the minimal probability that neuron $i$ does not spike. Next, for each $k \geq 1$ and each $\zeta \in \tilde S,$ we define the set $D_i^k(\zeta)$ by
$$
D_i^k(\zeta) := \{ \eta \in \tilde S: \forall j \in V_i(k), \eta({j \to i} )= \zeta({j \to i} ) \} 
$$ and put
\begin{equation}
r_i^{[k]} (1 | \zeta ) = \inf_{ \eta \in D_i^k(\zeta)   } \phi_i \left( \sum_j W_{ j \to i } \eta ({j \to i} ) \right), \quad \,
r_i^{[k]} (0 | \zeta ) = \inf_{ \eta \in D_i^k(\zeta)   } \left( 1- \phi_i \left( \sum_j W_{ j \to i } \eta ({j \to i}) \right) \right).
\end{equation}
Finally, we define 
$$ \alpha_i (0)  = r_i^{[0]} ( 1)  + r_i^{[0]} ( 0 ) , \quad \alpha_i (k)  = \inf_{\zeta \in \tilde S } \left( r_i^{[k]} (1 | \zeta )  +  r_i^{[k]} (0 | \zeta ) \right) $$
for all $k \geq 1$ and let
$$\mu_i (0) = \alpha_i (0) \mbox{ and } \mu_i(k)  =   \alpha_i (k )- \alpha_i (k-1)  $$
for all $k \geq 1.$ 

\begin{lem}\label{lem:muprobahc}
$\left( \mu_i(k)\right)_{k \geq 0}$ defines a probability on $\N. $ 
\end{lem}

\begin{proof}
Indeed, since $D_i^k(\zeta) \subset D_i^{k-1}(\zeta) ,$ we have $\mu_i(k) \geq 0$ for all $k \geq 0.$ Therefore, all we have to show is that
\begin{equation}\label{eq:proba}
\sum_{k \geq 0} \mu_i(k) =1.
\end{equation}
Note that 
$\sum_{k \geq 0} \mu_i(k) = \lim_{k \to + \infty} \left[ \inf_{\zeta \in \tilde S} \left( r_i^{[k]} (1 | \zeta)  +  r_i^{[k]} (0 | \zeta ) \right) \right].$
Hence it is sufficient to show that for all $\zeta \in \tilde S,$ 
$$\lim_{k \to + \infty}  \left( r_i^{[k]} (1 | \zeta )  +  r_i^{[k]} (0 | \zeta) \right)=1 .$$
For all $k \geq 0,$ we have
$$
r_i^{[k]} (1 | \zeta )  +  r_i^{[k]} (0 | \zeta ) 
= 1 - \left[ \sup_{ \eta \in D_i^k(\zeta)   } \phi_i \left( \sum_j W_{ j \to i } \eta ({j \to i} )\right)  -\inf_{ \eta \in D_i^k(\zeta)   } \phi_i \left( \sum_j W_{ j \to i } \eta ({j \to i} )\right) \right].
$$
Using that $\phi_i$ is Lipschitz with Lipschitz constant $\gamma $ and increasing, we deduce that
\begin{multline*}
 \sup_{ \eta\in D_i^k(\zeta)   } \phi_i \left( \sum_j W_{ j \to i } \eta ({j \to i})  \right) -\inf_{ \eta \in D_i^k(\zeta)   } \phi_i \left( \sum_j W_{ j \to i }  \eta ({j \to i} ) \right) 
 \\ 
 \leq \gamma \left[ \sup_{ \eta \in D_i^k(\zeta )   } \left( \sum_j W_{ j \to i } \eta ({j \to i} ) \right) -\inf_{ \eta \in D_i^k(\zeta)   } \left( \sum_j W_{ j \to i } \eta( {j \to i} ) \right) \right]
 \\
 \leq \gamma \sum_{j \notin V_i(k)} \left| W_{ j \to i } \right| K_{j \to i}.
\end{multline*}
Now, taking the limit as $k$ tends to $+ \infty $ and taking into account condition (\ref{eq:horriblehc}), we obtain (\ref{eq:proba}) as desired.
\end{proof}

Let us come back to the conditional probability $p_{(i,t)}(1|\F_t ) $ introduced in (\ref{eq:transition2hc}). The history is realized only through the effected choices of acceptance or rejection of jumps of the a priori PRM $N.$ Therefore we introduce the time grid $\cG = \{ (i, T_n^i), i \in I \} . $ Any realization of the Hawkes process, conditionally with respect to the PRM $N$, can be identified with an element of $X:=  \{0,1\}^{\cG} .$ We write $x= (x^i)_{i \in I }$ for elements of $X,$ where $x^i =( x^i ( T^i_n))_{n \in \Z}.$ Elements $x \in X$ can be interpreted as point measures. The object of our study is
$$ p_{(i,t)} ( 1 | x ) = \phi_i \left( \sum_{ j } W_{j \to i } \left( x^j  \left( [ L_t^i(x) , t [ \right) \wedge K_{j \to i} \right) \right) .$$
The following proposition establishes a Kalikow-type decomposition for $p_{(i, t)} ( \cdot | x) $ with respect to growing neighbourhoods of $\cV_{\cdot \to i }.$

\begin{prop}\label{prop:1hc}
Grant Assumption \ref{ass:1} and assume that (\ref{eq:summable}) and (\ref{eq:horriblehc}) are satisfied. Fix $t = T_n^i $ for some $n \in \Z $ and $i \in I .$ Then there exists a family of conditional probabilities $( p_{(i,t)}^{[k]} ( \cdot  |x ))_{ k \geq 0 } $ on $\{ 0,1 \} $ satisfying the following properties.
\begin{enumerate}
\item 
For all  $a \in \{ 0,1 \}, p_{(i,t) }^{[0]  } ( a|x):= \frac{r_i^{[0]} ( a)}{\mu_i (0)}$ does not depend on the configuration $x.$
\item
For all  $a \in \{ 0,1 \} ,$ $ k \geq 1 ,$  $ X \ni x \mapsto p_{(i,t)}^{[k] } ( a | x ) $ depends only on the variables $  x^j :  j \in V_i (k) .$ 
\item
For all $x \in  X , $ $ k \geq 1,$ $ p_{(i,t)}^{[k] } ( 1 | x ) \geq 0 $ and $ p_{(i,t)}^{[k] } ( 1 | x ) +p_{(i,t)}^{[k] } ( 0 | x )  = 1 .$
\item
For all $x \in X,$ we have the following convex decomposition.
\begin{equation}\label{eq:dec2hc}
p_{(i,t)} ( a | x )  = 
\sum_{ k \geq 0 } \mu_i(k) p_{(i,t)}^{[k] } ( a | x ) ,
\end{equation}
\end{enumerate}
\end{prop}

\begin{proof}
We identify a configuration $x \in X $ with an element $x_t $ of $\tilde S $ for any $t \in \R $ by introducing 
$$ x_t ( j \to i ) = x^j ( [L_t^i , t[ ) \wedge K_{j \to i } .$$
Note that $ p_{(i,t)} ( a | x ) $ only depends on $x_t.$ We have
\begin{equation}\label{eq:dec3hc}
p_{(i,t)} ( a | x ) = p_{(i,t)} ( a | x_t ) = r_i^{[0]} ( a) + 
\sum_{ k =1 }^N \Delta_i^{[k]}(a|x_t) + \left( p_{(i,t)} ( a | x_t ) - r_i^{[N]} ( a | x_t ) \right) ,
\end{equation}
where
$
\Delta_{i}^{[k]}(a|x_t) := r_i^{[k]} ( a | x_t) - r_i^{[k-1]} ( a | x_t).
$
We start by showing that the term $p_{(i,t)} ( a | x_t ) - r_i^{[N]} ( a | x_t )$ in (\ref{eq:dec3hc}) tends to $0$ as $N \to \infty .$ Indeed, rewriting the definitions, we have
$$
 \Big| p_{(i,t)} ( a | x_t ) - r_i^{[N]} ( a | x_t ) \Big| \\ = \sup_{z \in D_i^N(x_t) } \Big|  \phi_i  \left( \sum_j W_{j\to  i } x_t(j \to i) \right) - \phi_i  \left( \sum_{j } W_{j \to i}  z_{j \to i}  \right)  \Big|.
$$
Using the Lipschitz continuity of $\phi_i$ and the definition of $D_i^N (x_t),$ we get
\begin{eqnarray*}
\Big| p_{(i,t)} ( a | x_t ) - r_i^{[N]} ( a | x_t ) \Big|  &\leq &\sup_{z \in  D_i^N(x_t) } \left( \gamma   \sum_j | W_{j\to  i } | | x_t(j \to i) -  z_{j \to i} |  \right) \\
& \leq &\gamma \sup_{z \in \tilde S}  \left(  \sum_{j \notin V_i(N)} | W_{j\to  i } | | x_t(j \to i) - z_{j \to i} |  \right) 
 \\
& \leq & \gamma \left( \sum_{j \notin V_i(N)} | W_{j\to  i } | K_{j \to i} \right) \underset{N \to + \infty}{\longrightarrow} 0.
\end{eqnarray*}
Taking the limit when $N$ tends to $+ \infty$ in (\ref{eq:dec3hc}), we obtain
$$p_{(i,t)} ( a| x_t ) = r_i^{[0]} ( a) + \left( \sum_{ k \geq 1 } \Delta_{i}^{[k]} ( a | x_t ) \right).$$
Now we put for  $ k \geq 1$ 
\begin{equation}\label{eq:lambdaalmosthc}
 \tilde{\mu}_{i } (k, x_t) :=  \sum_a  \Delta_{i}^{[k]} ( a | x_t )  
\quad \mbox{and} \quad \tilde{p}_{(i,t)}^{[k]} (a |  x)= \frac{\Delta_{i}^{[k ]} ( a | x_t)}{ \tilde{\mu}_{i}  (k,  x_t) },
\end{equation}
where we define $\tilde{p}_{(i,t)}^{[k]} (a| x_t) $ in an arbitrary way if
$ \tilde{\mu}_{i}  (k,  x_t)= 0.$ In this way we can write $  \Delta_{i}^{[k ]} ( a | x_t)=  \tilde{\mu}_{i}  (k,  x_t) \tilde{p}_{(i,t)}^{[k]} (a |  x_t) $
and therefore
\begin{equation}\label{eq:almosthc}
p_{(i,t)} ( a | x_t ) = \mu_i (0) \, p_{(i,t)}^{[0]} (a) + \sum_{k=1}^{\infty} \tilde{\mu}_{i}  (k,  x_t)  \, \tilde{p}_{(i,t)}^{[k]} (a |  x_t ). 
\end{equation}
This decomposition is not yet the one announced in the proposition since the $\tilde{\mu}_{i}  (k,  x_t)$ depend on the configuration $x_t.$
The weights $ \mu_{(i,t)} (k) $ are already defined and they do not depend on the configuration. So we have to define new probabilities $ p_{(i,t)}^{[k]} ,$ based on the previously defined $\tilde p_{(i,t)}^{[k]}.$ % on the intervals  $ ( \alpha_{i}  (k-1) , \alpha_{i}  (k) ]. $ 
To do this, we introduce $
\alpha_{i} (k , x_t):= \sum_{l=1}^k \tilde{\mu}_{i}  (l,  x_t).$ For each $ k \geq 0 $ let $l'$ and $l$ be indexes such that 
$$\alpha_i (l'- 1, x_t ) < \alpha_i (k-1) \le \alpha_i (l', x_t )  < \ldots < \alpha_i (l, x_t )  < \alpha_i (k) \le \alpha_i (l +1, x_t ) .
$$
We then decompose the interval $ ] \alpha_i  (k-1) , \alpha_i  (k) ]$ in the following way.
\begin{multline*} 
] \alpha_i ( k- 1 ) , \alpha_i ( k ) ] = ] \alpha_i (k-1) , \alpha_i (l', x_t ) ] \cup \\
\left( \bigcup_{ m = l' + 1 }^l ] \alpha_i ( m-1 , x_t) , \alpha_i (m, x_t ) ] \right) \cup \; ] \alpha_i ( l, x_t) , \alpha_i ( k ) ] .
\end{multline*}
We then have to define $p_i^{[k]} $ on each of the intervals of this partition. On the first interval $ ] \alpha_i (k-1) , \alpha_i (l', x_t ) ] , $ we use $ \tilde p^{[l']}_i , $ on each interval $] \alpha_i ( m-1 , x_t) , \alpha_i (m, x_t ) ]  ,$ we use $ \tilde p_i^{ [m]} , $ and on  $  ] \alpha_i ( l, x_t) , \alpha_i ( k ) ] ,$ we use $ \tilde p_i^{ [ l+1 ] } .$ 
This leads, for each  $k \geq 0$, to the following definition.
\begin{multline}\label{eq:defprobadeco}
p_{(i,t)}^{[k]} ( a| x_t )  
=
 \sum_{-1 = l' \le l }^{k-1}  \indiq_{\{
\alpha_i (l' - 1 ,x_t) < \alpha_i ({k-1}) \le \alpha_i ({l'},
x_t)\}} 
\indiq_{\{
\alpha_i (l, x_t) < \alpha_i ({k}) \le \alpha_i ({l+1},
x_t)\}}  \\
\left[ \frac{\alpha_i (l', x_t ) - \alpha_i (k-1) }{ 
\mu_i ({k})} \; \tilde{p}_{(i,t)}^{[l']} (a | x )\right.   
  + \sum_{m = l'+1}^{l} \frac{\tilde{\mu}_i  (m , x_t)}{  \mu_i (k) }
\;  \tilde{p}_{(i,t)} ^{[m]} (a | x) 
  \\
\left. + \; \frac{\alpha_i ({k}) - \alpha_i  (l, x_t )}{
\mu_i ({k})} \; \tilde{p}_{(i,t)} ^{[l+1]} (a| x_t) \right] .
\end{multline} 
It can be easily verified that with this definition, we obtain the announced decomposition. 
%Indeed, writing $\sum_{ k \geq 0 } \mu_i(k) p_{(i,t)}^{[k] } ( a | x )$ with this definition leads to the simplification of the denominators $\mu_i(k)$ in  (\ref{eq:defprobadeco}) so that we obtain the terms of the form $\tilde{\mu}_i  (m , x_t) \tilde{p}_{(i,t)} ^{[m]} (a | x)$ present in (\ref{eq:almosthc}). To deal with the terms $\alpha_i (l', x_t ) - \alpha_i (k-1)$ and $\alpha_i ({k}) - \alpha_i  (l, x_t )$ we have to pair them up according to indexes $l'$ and $l$ such that $l'=l+1.$ By this mean we can simplify the terms $\alpha_i ({k})$ to obtain $\tilde{\mu}_i  (l' , x_t) \tilde{p}_{(i,t)} ^{[l']} (a | x)$ and use (\ref{eq:almosthc}) to conclude the proof.
\end{proof}

\subsection{Perfect simulation}\label{subsec:pshc}
In this section we show how to construct the stationary nonlinear Hawkes process with saturation threshold by a {\it perfect simulation procedure}, based on an a priori realization of the  
processes $(N^i, i \in I ).$ Condition (\ref{eq:deltahc}) allows us to decompose the Poisson process ${N}^i$ of intensity $\Lambda_i$ as 
$$
N^i= \widehat{N}^i + \widetilde{N}^i,
$$
where $\widehat{N}^i$ and $\widetilde{N}^i$ are independent Poisson processes with respective intensities $\delta_i$ and $\Lambda_i - \delta_i.$
Conditionally on these processes, we can characterize the process $Z^i$ by the element $x \in X= \{0, 1 \}^{\cG}$ recording the times and the indices of the neuron for which we have accepted a jump. All jumps of $\widehat{N}^i$ are also jumps of $Z^i ;$ we call them the {\it spontaneous spikes}. They appear at a jump time $T_n^i$ of $N^i$ with probability $d_i:= \frac{\delta_i}{\Lambda_i}.$ Moreover, any jump $\widetilde{T_n^i}$ of $\widetilde{N}^i$ will be a jump of $Z^i$ with probability
$$
\frac{1}{1-d_i} \left[ \phi_i  \left( \sum_j  W_{ j \to i} \left( Z^j \left( [ L_t^i , t [ \right) \wedge K_{j \to i} \right) \right) - d_i \right],
$$
and we have to decide for each $i \in I$ and each time $\widetilde{T}_n^i$ whether this jump is accepted or not.
This acceptance/rejection procedure will be achieved by means of the Kalikow-type decomposition and gives rise to a perfect-simulation algorithm that we are going to introduce now. 

%%%%%%%%JUSQU'ICI REPRIS PAR EVA %%%%%%%%%%%%%%%%%%%%%%%%

Fix $i \in I$ and $ t  = \widetilde T_n^i .$  Our first algorithm describes the random space-time subset of neurons and their associated spiking times that can possibly have an influence on the acceptance or rejection of a spike of neuron $i$ at time $t.$ 
\begin{enumerate}
\item We simulate all $\widehat{N}^j$ and $\widetilde N^j$ for each $j \in I.$ We will use the realization of $\widehat{N}^j$ to determine the last spontaneous spiking time 
$$\widehat  L_s^j=\widehat T_{\widehat N_{s-}^j}^j,$$ 
of neuron $j$ before time $s. $ The processes $\widetilde N^j$ give us the non-spontaneous spiking times that we have to accept or reject, which is the aim of the present algorithm.
\item
We attach i.i.d.\ random variables $ K^{(j,s) } \in \N $ to each couple $ (j, s  ) $ with $ s = \widetilde T_m^j  $ for some $m$ which are chosen according to 
$$ \P( K^{(j,s) } = k ) = \mu_j ( k ) , \mbox{ for all } k \geq 0.$$
\item
Suppose that $K^{(i,t) } = k .$ We introduce the set 
\begin{equation}\label{eq:defC1}
C^{(i,t)}_1:=  \left\{  (j,\widetilde T_l^j) \in I \times ]-\infty , t[: j \in V_i(k), \widetilde T_l^j \in [ \widehat L_{t}^i,t [\right\},
\end{equation}
which we call ``first generation of the clan of ancestors" of element $(i,t).$ Here, by convention, $C^{(i,t)}_1 = \emptyset $ if $K^{(i,t) }  = 0.$ 
\item
If $ C^{(i,t)}_1 \neq \emptyset, $ then we iterate the above procedure and simulate for each element of  $C^{(i,t)}_1$ a new clan of ancestors. So we put for any $ n \geq 2,$ 
\begin{equation}\label{eq:defCn}
 C_n^{(i,t)} = \left( \bigcup_{ (j, s) \in C_{n-1}^{(i,t)}} C^{(j,s)}_1  \right) \setminus \left( C_1^{(i,t)} \cup \ldots \cup C_{n-1}^{(i,t)}\right),
\end{equation}
which is the ``$n-$th generation of the clan of ancestors'' of $(i,t).$ 
\end{enumerate}
Notice that if $K^{(i,t) } =0$ in step (1), then the algorithm stops immediately. Also if for all $j \in V_i(k), $ for $k = K^{(i,t)} , $ we have $ \widetilde N^j ( [ \widehat L_{t}^i,t [ ) = 0 ,$ then $C^{(i,t)}_1 = \emptyset  ,$ and the algorithm stops after one step. Else we introduce 
$$ N^{Stop} := \min \{ n : C_n^{(i,t) } = \emptyset \} $$
the number of steps of the algorithm, where $ \min \emptyset := \infty.$ The set $\mathcal C^{(i,t)} := \bigcup_{n=1}^{N^{Stop}} C_n^{(i,t)} $ contains all non-spontaneous spikes which have to be accepted or not and whose possible presence has an influence on the acceptance/rejection decision of $(i,t) $ itself. We will show below that under the conditions of Theorem \ref{theo:1hc}, $N^{Stop} < \infty $ almost surely.  

Once the clans of ancestors are determined, we can realize the acceptance/rejection procedure of the elements in these clans in a second algorithm which is a forward procedure going from the past to the present. We start with the sites for which this decision can be independently made from anything else. During the algorithm,  the set of all sites for which a decision has already been achieved, will then progressively be updated.

\begin{enumerate}

\item At the initial stage of the algorithm, the set of sites for which the acceptance/rejection decision can be achieved is initialized by  
$$D^{(i,t)}:= \left\lbrace (j,s) \in \mathcal C^{(i,t)} , C_1^{(j,s)}= \emptyset \right\rbrace .$$
The sites within this set are precisely those for which the decision can be independently made from anything else. 
\item For each $(j,s) \in D^{(i,t)},$ we simulate, according to the probabilities $ \frac{1}{1- \delta} \left( p_{(j,s)}^{[0]} (a) - \delta \right), $ the state of this site. 
\item For any $(j,s) \in \mathcal C^{(i,t)}$ with $C_1^{(j,s)} \subset D^{(i,t)},$ we then decide, according to the probabilities $ \frac{1}{1- d_i} \left( p_{(j,s)}^{[k]} (a|x) - d_i \right) $ with $k=K^{(j,s)},$ to accept or to reject the presence of a spike of neuron $j$ at time $s.$ Then we update $D^{(i,t)},$ which in an algorithmic way can be written $$D^{(i,t)} \longleftarrow D^{(i,t)} \bigcup \left\lbrace (j,s) \in \mathcal C^{(i,t)}, C_1^{(j,s)} \subset D^{(i,t)} \right\rbrace.$$
\item The update of $D^{(i,t)}$ allows us to repeat the previous step until $(i,t) \in D^{(i,t)}.$ 
\end{enumerate}
Once we have assigned a decision to the element $ (i,t ) $ itself, our {\it perfect simulation} algorithm stops.  Of course, the whole procedure makes sense only if $N^{Stop} < + \infty$ a.s.\ which we will prove now.  For that sake, we define $C^{(i,t)} (k)$ to be the clan of ancestors of element $(i,t)$ of size $k,$ i.e.\
$$
C^{(i,t)}(k):=\left\lbrace (j ,s) \in I \times \R : j \in V_i(k), \exists n \in \Z \mbox{ s.t. } s= \widetilde   T_n^i, s \in \left[ \widehat{L}_t^i;t \right) \right\rbrace ,
$$
with the convention $C^{(i,t)} (0) := \emptyset $ and put
$$ M^{(i, t)} :=  \sum_{ k \geq 1 } \left| C^{(i,t)} (k) \right| \mu_i ( k ) $$
which is the conditional expectation, with respect to the PRM $N,$ of $\big|C_1^{(i,t)} \big|.$ In order to prove that $N^{Stop} < + \infty$ a.s., we compare the process $ | C_n^{(i,t)} | $  with a branching process of reproduction mean $M$ such that $$ \mathbb E \left( M^{(i, t)} \right) \leq M.$$ 
We will prove that the parameters of this branching process are such that it goes extinct a.s.\ implying that $N^{Stop} < \infty $ almost surely. 

Writing $\mathbb E^N$ for the conditional expectation with respect to $N,$ we obtain the following recurrence. 
\begin{equation}\label{eq:214hc}
 \mathbb E^N (  | C_n^{(i,t)} | | C_{n-1}^{ (i,t)} )  \le \sum_{ (j, s ) \in C_{n-1}^{(i,t)} } \sum_{k=1}^{+ \infty} \sharp \left( C^{(j,s)}(k) \backslash C_{n-1}^{(i,t)} \right) \mu_i(k) .
\end{equation}
Here we use $\sum_{k=1}^{+ \infty} \sharp \left( C^{(j,s)}(k) \backslash C_{n-1}^{(i,t)} \right) \mu_i(k)$ instead of $ M^{(j,s)}$ in order to use the independence of $1_{ (j ,s) \in C_{n-1}^{(i,t)} }$ and $\sum_{k=1}^{+ \infty} \sharp \left( C^{(j,s)}(k) \backslash C_{n-1}^{(i,t)} \right) \mu_i(k).$ This independence is due to the properties of the Poisson Random Measure associated with $N$ and the fact that $ C_{n-1}^{(i,t)}$ and $ C_{(j,s)}(k) \backslash C_{n-1}^{(i,t)} $ are disjoint. It will allow us to claim

\begin{equation}\label{eq:claimhc}
\mathbb E \left( \sum_{(j, s ) } 1_{ (j ,s) \in C_{n-1}^{(i,t)} } \sum_{k=1}^{+ \infty} \sharp \left( C^{(j,s)}(k) \backslash C_{n-1}^{(i,t)} \right) \mu_i(k) \right)
 \leq \sum_{(j, s ) } \mathbb E ( 1_{ (j ,s) \in C_{n-1}^{(i,t)}}  ) \mathbb E \left( M_{(j,s)} \right).
\end{equation}
In order to prove the above inequality in a rigorous way, we study the transition governing the evolution \eqref{eq:defCn}. This leads to the definition of the following transition kernel. We fix a neuron $i \in I,$ a time $t= \widetilde T_n^i$ and put
\begin{multline*}
Q \left( (i,t), \bullet \right) = \sum_{k \geq 0} \mu_i(k) \int_0^{+ \infty}  dt_1 \left( \delta_i e^{- \delta_i t_1} \right) \\ \left[ \sum_{j \in V_i(k)} \left( \sum_{n_j \geq 0} e^{ - \left( \Lambda_j - \delta_j \right) t_1} \frac{\left( \Lambda_j - \delta_j \right)^{n_j}}{n_j !} \int_{[t- t_1,t]^{n_j}} d s_1^j ... d s_{n_j}^j \delta_{ \{ (j,s_l^j): l=1,...,n_j \} }  \right) \right].
\end{multline*}
$Q \left( (i,t), dC_1 \right)$ is the law of $C_1^{(i,t)}.$ In the above definition, $k$ is the ``size" of the neighbourhood $V_i(k)$ which is simulated according to the probabilities $\left( \mu_i(k) \right)_{k \in \N };$ $t_1$ is the time between $t$ and $\widehat{L}_t^i$  which is an exponential random variable of parameter $\delta_i;$ $n_j$ is the Poisson random variable $\widetilde N^j \left( \left[ \widehat{L}_t^i , t \right] \right)$ of parameter $\left( \Lambda_j - \delta_j \right) t_1;$ and $\left( s_l^j \right)_{l \in \{ 1,...,n_j \} }$ is the family of jump times in $ \left[ \widehat{L}_t^i , t \right]$ of $\widetilde N^j,$ which, conditionally on the event $\{ \widetilde N^j \left( \left[ \widehat{L}_t^i , t \right] \right) = n_j \},$ are uniform random variables on $\left[ \widehat{L}_t^i , t \right] = [t-t_1,t].$

The definition of the transition kernel $Q(C_1,dC_2)$ defining the law of $C_2^{(i,t)}$ knowing $C_1^{(i,t)}$ is more complicated since $C_2$ is not the result of independent simulations for each element of $C_1.$\footnote{Indeed,  the jump times of a process $\widetilde N^j$ simulated for the clan of ancestors of an element $(i_1,t_1)$ of $C_1$ have to be simulated once and for all, and re-used for the determination of the clan of ancestors of an element $(i_2,t_2)$ with $j \in V_{i_1} (k_1) \cap  V_{i_2} (k_2).$ Similarly, the jump times of $\widehat{N}^i$ have to be simulated once and for all in order to determine all the last spontaneous spiking times for all $(i,t) \in C_1.$} But $Q(C_1,dC_2)$ can be upper bounded (in the sense of inclusion of the simulated sets) by 
$$ \prod_{ (j,s) \in C_1} Q( (j,s) , d C^{(j,s)} )\delta_{( \cup_{ ( j,s) \in C_1} C^{(j, s ) })}( d C_2) .$$
This upper bound simulates more random variables than necessary leading to bigger clans of ancestors. Since we are only interested in obtaining upper-bounds on the number of elements in the clans of ancestors, we will therefore work with this upper bound and obtain
\begin{multline}\label{eq:cores}
\mathbb E \left( \sum_{(j, s ) } 1_{ (j ,l) \in C_{n-1}^{(i,t)} } \sum_{k=1}^{+ \infty} \sharp \left( C^{(j,s)}(k) \backslash C_{n-1}^{(i,t)} \right) \mu_i(k) \right) \\
 \le  \int_{(i,t)^c} Q((i,t),dC_1) \int_{C_1^c} Q(C_1,dC_2) ... \int_{\left( \cup_{\kappa = 1}^{n-2} C_\kappa \right)^c} Q(C_{n-2},dC_{n-1})
 \\ \left[ \sum_{(j,s) \in C_{n-1}} \int_{\left( \cup_{\kappa = 1}^{n-1} C_\kappa \right)^c} Q((j,s),dC) |C| \right] .
\end{multline}
Observe that
$$
\int_{\left( \cup_{\kappa = 1}^{n-1} C_\kappa \right)^c} Q((j,s),dC) |C| \leq \sum_{k \geq 0} \mu_j(k) \int_0^{+ \infty} dt_1 \left[ \left( \delta_j e^{- \delta_j t_1} \right) \sum_{a \in V_j(k)} \left( \left( \Lambda_a - \delta_a \right)t_1 \right) \right].
$$
The right term of the above inequality being exactly $\mathbb E \left( M^{(j,s)} \right),$ this proves claim (\ref{eq:claimhc}). 

As a consequence, we can compare the process $ | C_n^{(i,m)} | $  with a branching process of reproduction mean $ M .$ 
It remains to verify that this branching process gets extinct almost surely in finite time.  We will prove that actually $$ \sup_{i \in I, t \in \R } \mathbb E ( M^{(i,t) }) < 1.$$

\begin{prop}

\begin{equation}\label{eq:ubmuhc}
\mu_i (k) \leq \gamma \left( \sum_{j \in \partial V_i(k-1)} \left| W_{ j \to i } \right| K_{j \to i} \right), \quad \forall k \geq 1.
\end{equation}

\end{prop}

\begin{proof}

Using the definition of $\mu_i(k),$ 
$$
\mu_i (k) = \inf_{\zeta \in {\tilde S}} \left( r_i^{[k]} (1 | \zeta )  +  r_i^{[k]} (0 | \zeta ) \right) - \inf_{\zeta \in {\tilde S}} \left( r_i^{[k-1]} (1 | \zeta )  +  r_i^{[k-1]} (0 | \zeta ) \right) .
$$
Fix $\varepsilon > 0$ and let $\eta \in \tilde S$ be such that 
$$
 r_i^{[k-1]} (1 | \eta )  +  r_i^{[k-1]} (0 | \eta ) \leq \inf_{\zeta \in {\tilde S}} \left( r_i^{[k-1]} (1 | \zeta )  +  r_i^{[k-1]} (0 | \zeta ) \right) + \varepsilon .
$$ 
Then
$$
\mu_i (k) \leq  \left( r_i^{[k]} (1 | \eta )  +  r_i^{[k]} (0 | \eta ) \right) -  \left( r_i^{[k-1]} (1 | \eta )  +  r_i^{[k-1]} (0 | \eta ) \right)  + \varepsilon
$$
%Here we can assume, without loss of generality that $L_{\bar T_m^i}^i (u) = - \infty$ because if $L_{\bar T_m^i}^i (u) > - \infty,$ let $u'$ be such that $L_{\bar T_m^i}^i (u')= - \infty, u'=0 $ before $L_{\bar T_m^i}^i (u)$ and $u'=u$ after, then $u'$ gives an equivalent configuration.
and
\begin{multline*}
\mu_i (k) - \varepsilon \leq \inf_{ z\in D_i^k(\eta)   } \phi_i \left( \sum_j W_{ j \to i } z_{j \to i}  \right)
- \inf_{ z\in D_i^{k-1}(\eta)   } \phi_i \left( \sum_j W_{ j \to i } z_{j \to i}  \right)
\\ +  \sup_{ z\in D_i^{k-1}(\eta)   } \phi_i \left( \sum_j W_{ j \to i } z_{j \to i}  \right)
- \sup_{ z\in D_i^k(\eta)   }\phi_i \left( \sum_j W_{ j \to i } z_{j \to i}  \right).
\end{multline*}
We will simplify this expression detailing the configurations that realize the extrema.  Since $\phi_i$ is non-decreasing, we have to study the upper and lower bounds of $\sum_j W_{ j \to i } z_{j \to i}$ for $z \in D_i^k(\eta).$
Recall that, due to the definition of $D_i^k(\eta),$ $z_{j \to i}$ is equal to $\eta_{j \to i}$ if $j \in V_i(k).$ Thus the extrema of $\sum_j W_{ j \to i } z_{j \to i}$ are reached for extreme values of $z_{j \to i}$ for $j \notin V_i(k)$ and these extreme values are $0$ or $K_{j \to i}.$ The upper bound is reached for $z_{j \to i}=K_{j \to i}$ with $j \notin V_i(k)$ such that $W_{j \to i} > 0$ and $z_{j \to i}=0$ with $j \notin V_i(k)$ such that $W_{j \to i} \leq 0.$ On the contrary, the lower bound is reached for $z_{j \to i}=K_{j \to i}$ with $j \notin V_i(k)$ such that $W_{j \to i} < 0$ and $z_{j \to i}=0$ with $j \notin V_i(k)$ such that $W_{j \to i} \geq 0.$ 

Using condition (\ref{eq:Lip}) we get therefore
$$
\mu_i (k)- \varepsilon \leq  \gamma \left[ \left( \sum_{j \notin V_i(k-1)} \left| W_{ j \to i } \right| K_{j \to i} \right) - \left( \sum_{j \notin V_i(k)} \left| W_{ j \to i } \right| K_{j \to i} \right) \right]
$$
so that finally 
$$
\mu_i (k) \leq \gamma \left( \sum_{j \in \partial V_i(k-1)} \left| W_{ j \to i } \right| K_{j \to i} \right)
$$
which is the announced upper-bound (\ref{eq:ubmuhc}).
\end{proof}
We are now able to finish the proof of Theorem \ref{theo:1hc}. We have the following first upper-bound for  $M^{(i,t) }$ thanks to the inequality (\ref{eq:ubmuhc})
$$ \frac{M^{(i,t) }}{\gamma}  \leq \sum_{k \geq 1} \left| C^{(i,t)} (k) \right| \left( \sum_{j \in \partial V_i(k-1)} \left| W_{ j \to i } \right| K_{j \to i} \right).
$$
Consequently, by definition of $C^{(i,t)}(k),$ 
\begin{eqnarray*}
\frac{\mathbb E \left( M_{(i,t)}\right)}{\gamma} 
&\leq& \sum_{k \geq 1} \mathbb E \left( \sum_{j \in V_i(k)} \widetilde N^j \left( \left[ \widehat{L}_t^i, t \right] \right) \right) \left( \sum_{j \in \partial V_i(k-1)} \left| W_{ j \to i } \right| K_{j \to i} \right) \\
&=& \frac{\mathbb E \left( M_{(i,t)}\right)}{\gamma} 
\leq \sum_{k \geq 1} \left[ \left( \sum_{j \in V_i(k)} \frac{\Lambda_j-\delta_j}{\delta_i} \right) \left( \sum_{j \in \partial V_i(k-1)} \left| W_{ j \to i } \right| K_{j \to i} \right) \right].
\end{eqnarray*}
Assumption (\ref{eq:horriblehc}) of Theorem 1 implies that $ M := \sup_i \mathbb E\left( M_{(i,t)}\right) < 1 .$ Consequently, 
$$\mathbb E (| C_n^{(i,t)} |) \le M^n \to 0 \mbox{ when } n \to \infty , $$
implying that the process $\left( C_n^{(i,t)} \right)_{n \in \N}$ goes extinct almost surely in finite time. As a consequence, $N^{Stop}< \infty $ almost surely and the perfect simulation algorithm stops after a finite number of steps.  This achieves the proof of the construction of the process.
\begin{flushright}
$\square$
\end{flushright}

\section{Proof of Theorem  \ref{theo:1} .}

\subsection{Some comments}
We now give the proof of Theorem \ref{theo:1}. As in the proof of Theorem \ref{theo:1hc} we use a Kalikow-type decomposition of the transition probabilities \eqref{eq:transition2hc}. We will use the same notations for objects with slightly different definitions but playing the same role in the proof, in order to avoid too heavy notations.

The main difference between the two models is that in the first one, no leakage term is present, whereas the second model contains a leak term through the functions $g_j.$  
Moreover, in the first model, the presence of thresholds allows us to obtain a Kalikow-type decomposition according to space (and not to time) with probabilities $\mu_i(k)$ that are deterministic and do not depend on the realization of the Poisson Random Measure $N.$ This is crucial because it gives us an independence argument in order to obtain statement (\ref{eq:claimhc}).  However, the statement of Theorem \ref{theo:1hc} is at the cost of two assumptions. The first one is the presence of spontaneous spikes due to condition (\ref{eq:deltahc}). The second assumption is the existence of the thresholds $K_{j \to i}$. 

In the second model, we will introduce a space-time decomposition in random environment with probabilities $\mu_{(i,t)}(k)$ that are $\sigma \left( N \right)$-measurable random variables. The independence argument leading to (\ref{eq:claimhc}) will here be ensured by the condition (\ref{eq:structure}) that we impose on the structure of the neural network.

\subsection{Kalikow-type decomposition}

In this section, the Kalikow-type decomposition will take place in a random environment depending on the realization of the Poisson Random Measure $N.$ We will consider space-time neighbourhoods $\mathbb  V_t^i(k) := V_i(k) \times [t-k,t[. $ 

We work with the state pace $X = \{0, 1 \}^{\mathcal G} $ where ${\mathcal G} $ is the time grid $ \{(i,T_n^i),(i,n) \in I \times \Z \} .$ Note that each element $x$ of $X$ can be interpreted as a discrete measure dominated by $dN.$ 
We fix a time $t=T_n^i$ and we introduce the following notations
\begin{equation}
r_{(i,t)}^{[0]} ( 1) = \inf_{ x \in X} \phi_i \left( \sum_j W_{ j \to i } \int_{ L_t^i (x) }^t g_j (t-s) dx_j(s)  \right) ,
\end{equation}
where $dx_j = \sum_{m \in \Z} x \left(j,T_m^j \right)  \delta_{T_m^j},$ $ L_t^i ( x) = \sup\{ T_m^i < t : x ( i, T_m^i ) = 1 \} ,$  and
\begin{equation}
r_{(i,t)}^{[0]} ( 0) = \inf_{ x \in {\mathcal S}} \left( 1-\phi_i \left( \sum_j W_{ j \to i } \int_{ L_t^i (x) }^t g_j (t-s) dx_j(s)  \right) \right).
\end{equation}
Now fix $x \in X$ and define for each $k \geq 1$ the set $D_{(i,t)}^k(x)$ by
$$
D_{(i,t)}^k(x) := \{ z \in {\mathcal S} : z \left( \mathbb V_t^i ( k)  \right) = x \left( \mathbb V_t^i ( k)  \right) \} 
$$ and put
\begin{equation}
r_{(i,t)}^{[k]} (1 | x ) = \inf_{ z\in D_{(i,t)}^k(x)   } \phi_i \left( \sum_j W_{ j \to i } \int_{ L_t^i (z) }^t g_j (t-s) dz_j(s) \right) ,
\end{equation}
\begin{equation}
r_{(i,t)}^{[k]} (0 | x ) = \inf_{ z\in D_{(i,t)}^k(x)   } \left( 1- \phi_i \left( \sum_j W_{ j \to i } \int_{ L_t^i (z) }^t g_j (t-s) dz_j(s) \right) \right).
\end{equation}
We define 
$$ \alpha_{(i,t)} (0) := \mu_{(i,t)} ( 0 ) := r_{(i,t)}^{[0]} ( 1)  + r_{(i,t)}^{[0]} ( 0 ) , \quad \alpha_{(i,t)}^{[k]} : = \inf_{x \in {\mathcal S}} \left( r_{(i,t)}^{[k]} (1 | x )  +  r_{(i,t)}^{[k]} (0 | x ) \right) ,$$
and finally, for all $k \geq 1$ 
$$ \mu_{(i,t)}(k)  :=   \alpha_{(i,t)}^{[k]} - \alpha_{(i,t)}^{[k-1]} . $$
These definitions are similar to those of Section 4.1 and play the same role in the proof. Note that in this section we have a dependence on time denoted by the index $t.$

\begin{lem}
$\left( \mu_{(i,t)}(k)\right)_{k \geq 0}$ defines a probability on $\N. $ 
\end{lem}

\begin{proof}
With similar arguments to those of the proof of Lemma \ref{lem:muprobahc} and considering the worst case where $L_t^i (x) = - \infty,$ we obtain that 
\begin{multline}\label{eq:ubrest}
0 \leq 1- \left( r_{(i,t)}^{[k]} (1 | x ) + r_{(i,t)}^{[k]} (0 | x ) \right) \\ \leq \gamma \left( \sum_{j \in V_i(k) } | W_{j\to  i } | \int_{ -\infty }^{t-k} | g_j (t-s) | \, d N_s^j  + \sum_{j \notin V_i(k) } | W_{j\to  i } | \int_{ -\infty }^t | g_j (t-s) | \, d N_s^j \right).
\end{multline}
The right term of this inequality is a non-increasing sequence of positive random variables where each term is upper-bounded by a random variable of expectation
$$
\sum_{j \in V_i(k) } | W_{j\to  i } | \Lambda_j \int_{ k }^{+\infty} | g_j (s) | \, d s + \sum_{j \notin V_i(k) } | W_{j\to  i } | \Lambda_j \int_{ 0 }^{+\infty} | g_j (s) | \, ds.
$$
Due to condition (\ref{eq:horrible}) the previous expression goes to $0$ when $k$ goes to $+ \infty$ and this gives us the almost sure convergence of the sequence to $0$ due to the monotonicity of the sequence. Consequently, we have $\mu_{(i,t)}(k) \geq 0$ for all $k \geq 0$ and  
$$
\sum_{k \geq 0} \mu_{(i,t)}(k) =1.
$$
\end{proof}
Using these probabilities $\left( \mu_{(i,t)}(k) \right)_{k \geq 0},$ we obtain a Kalikow-type decomposition for
$$ p_{(i,t)} ( 1 | x ) = \phi_i \left( \sum_{ j } W_{j \to i } \int_{ L_t^i (x) }^t g_j (t-s) dx_j(s)  \right) .$$
Recall that $ \F_t $ is the sigma-field generated by $ Z^i ( ] s, u ] ) , s \le u \le t , i \in I,$ and that we work on the probability space $(\Omega , \mathcal A , \mathbb F)$ defined in Section \ref{subsec:setting}.

\begin{prop}\label{prop:1} 
Assume that conditions (\ref{eq:Lip}), (\ref{eq:summable}) and (\ref{eq:horrible}) are satisfied. Then there exists  a family of conditional probabilities $\left( p_{(i,t)}^{[k]} ( a |x ) \right)_{ k \geq 0 } $ satisfying the following properties.
\begin{enumerate}

\item 
For all  $a \in \{ 0,1 \}, p_{(i,t) }^{[0]  } ( a|x):= \frac{r_{(i,t)}^{[0]} ( a)}{\mu_{(i,t)} (0)}
$ does not depend on the configuration $x.$
\item
For all  $a \in \{ 0,1 \} ,$ $ k \geq 1 ,$  $ {\mathcal S} \ni x \mapsto p_{(i,t)}^{[k] } ( a | x ) $ depends only on the variables \\ $ \left( x ( j, T_n^j ): (j,T_n^j) \in \mathbb V_t^i (k) \right) .$ 
\item
For all $x \in  {\mathcal S} , $ $ k \geq 1,$ $ p_{(i,t)}^{[k] } ( 1 | x ) \in [0, 1 ], $ $ p_{(i,t)}^{[k] } ( 1 | x ) +p_{(i,t)}^{[k] } ( 0 | x )  = 1 .$
\item
For all $a \in \{ 0,1 \}, x \in {\mathcal S}, p_{(i,t)}^{[k] } ( a | x )$ and $ \mu_{(i,t)}(k) $ are $ {\mathcal F}_t -\mbox{measurable} $  random variables.

\item
For all $x \in {\mathcal S},$ we have the following convex decomposition
\begin{equation}\label{eq:dec2}
p_{(i,t)} ( a | x )  = 
\sum_{ k \geq 0 } \mu_{(i,t)}(k) p_{(i,t)}^{[k] } ( a | x ) ,
\end{equation}

\end{enumerate}
\end{prop}

\begin{proof}
As in the proof of Proposition \ref{prop:1hc}, we write
\begin{equation}\label{eq:dec3}
p_{(i, t)} ( a | x )  = r_{(i,t)}^{[0]} ( a) + 
\sum_{ k =1 }^N \Delta_{(i,t)}^{[k]}(a|x) + \left( p_{(i,t)} ( a | x ) - r_{(i,t)}^{[N]} ( a | x ) \right) ,
\end{equation}
where
$$
\Delta_{(i,t)}^{[k]}(a|x) := r_{(i,t)}^{[k]} ( a | x ) - r_{(i,t)}^{[k-1]} ( a | x ).
$$
The first step of the proof consists in proving that $\left( p_{(i,t)} ( a | x ) - r_{(i,t)}^{[N]} ( a | x ) \right)$ 
 tends to $0$ as N tends to $+ \infty.$ Using similar arguments as in Proposition \ref{prop:1hc}, we obtain
 \begin{multline*}
 \Big| p_{(i,t)} ( a | x ) - r_{(i,t)}^{[N]} ( a | x ) \Big| \\ \leq \gamma \left( \sum_{j \in V_i(N) } | W_{j\to  i } | \int_{ -\infty }^{t-N} | g_j (t-s) | \, d N_s^j  + \sum_{j \notin V_i(N) } | W_{j\to  i } | \int_{ -\infty }^t | g_j (t-s) | \, d N_s^j \right).
\end{multline*}
We obtain the same upper bound as in (\ref{eq:ubrest}) converging almost surely to $0$ for the same reasons. The remainder of the proof is exactly identical to the proof of Proposition \ref{prop:1hc}.
\end{proof}

\subsection{Perfect simulation}
The idea for the perfect simulation algorithm is the same as in Section \ref{subsec:pshc}, but here we use a decomposition in space and time. We work conditionally on the realization of the Poisson Random Measure $N$ and consider, for each $k \geq 1,$ the space time neighbourhood $\mathbb V_t^i(k).$ Define the clan of ancestors of element $(i,t)$ of size $k$ by 
$$
C^{(i,t)}(k):= \mathcal G \cap \mathbb V_t^i(k),$$ 
where ${\mathcal G} $ is the time grid $ \{(i,T_n^i),(i,n) \in I \times \Z \} $ and where by convention $C^{(i,t)}(0):= \emptyset.$
We choose as before i.i.d.\ random variables $K^{(i,t)}\in \N$ which are attached to each site $(i,t) \in \mathcal G ,$ chosen according to 
$$ \P( K^{(i,t) } = k ) = \mu_{(i,t)} ( k ) , \mbox{ for all } k \geq 0.$$
These random variables allow us to define the clans of ancestors  $ (C_n^{(i,t)})_n  \subset I \times ]- \infty , t[ $ as follows. We put $C_1^{(i,t)} := C^{(i,t)} \left( K^{(i,t)} \right)$ and 
$$
 C_n^{(i,t)} := \left( \bigcup_{ (j, s) \in C_{n-1}^{(i,t)}} C^{(j,s)}_1  \right) \setminus \left( C_1^{(i,t)} \cup \ldots \cup C_{n-1}^{(i,t)}\right).
$$
As before, we have to prove that the process $ | C_n^{(i,t)} | $ converges almost surely to $0$ as $n$ tends to $+\infty .$ For this sake, we compare the process $ | C_n^{(i,t)} | $ with a branching process of reproduction mean (depending on space and time)
$$ M^{(i, t)} :=  \sum_{ k \geq 1 } \left| C^{(i,t)} (k) \right| \mu_{(i,t)} ( k )  . $$
Clearly,
\begin{equation}\label{eq:214}
\mathbb E^N (  | C_n^{(i,t)} | | C_{n-1}^{ (i,t)} )  \le \sum_{ (j, s ) \in C_{n-1}^{(i,t)} } M^{ (j,s ) } .
\end{equation}
Now, the structure that we imposed on the neural network in (\ref{eq:structure}) implies that for neurons $i \in I_l$ and $j \in I_m$ the event $\{ (j ,s) \in C_n^{(i,t)} \} $ is empty if $l \neq m + n.$ Moreover, this event depends only of realizations for neurons in layers $I_p$ with $m \leq p < l$ whereas $  M^{(j,s) } $ depends on realizations for neurons in the layer $I_{m-1}.$  Consequently $1_{ (j ,s) \in C_n^{(i,t)} }$ is independent of $  M^{(j,s) } ,$
which allows us to write
\begin{equation}\label{eq:claim}
\mathbb E \left( \sum_{(j, s ) } 1_{ (j ,s) \in C_n^{(i,t)} }M^{(j,s)} \right) = \sum_{(j, s ) } \mathbb E ( 1_{ (j ,s) \in C_n^{(i,t)}}  ) \mathbb E ( M^{(j,s) }).
\end{equation}
The next step in the proof is to show that $$ \sup_{i \in I} \mathbb E^N ( M^{(i,t) }) < 1.$$
As in Section \ref{subsec:pshc}, we start with a proposition that gives an upper bound for the probabilities $\left( \mu_{(i,t)} (k) \right)_{k \geq 1}.$

\begin{prop}

\begin{equation}\label{eq:ubmu}
\mu_{(i,t)} (k) \leq \gamma \left( \sum_{j \in V_i(k-1)} \left| W_{ j \to i } \right| \int_{ t-k }^{t-k+1} g_j (t-s) dN_s^j + \sum_{j \in \partial V_i(k-1)} \left| W_{ j \to i } \right| \int_{ t-k }^t g_j (t-s) dN_s^j \right).
\end{equation}

\end{prop}

\begin{proof}
Using the definition of $\mu_{(i,t)}(k),$
$$
\mu_{(i,t)} (k) = \inf_{x \in {\mathcal S}} \left( r_{(i,t)}^{[k]} (1 | x )  +  r_{(i,t)}^{[k]} (0 | x ) \right) - \inf_{x \in {\mathcal S}} \left( r_{(i,t)}^{[k-1]} (1 | x )  +  r_{(i,t)}^{[k-1]} (0 | x ) \right) .
$$
Fix $\varepsilon > 0$ and let $u \in {\mathcal S}$ be such that 
$$
 r_{(i,t)}^{[k-1]} (1 | u )  +  r_{(i,t)}^{[k-1]} (0 | u ) \leq \inf_{x \in {\mathcal S}} \left( r_{(i,t)}^{[k-1]} (1 | x )  +  r_{(i,t)}^{[k-1]} (0 | x ) \right) + \varepsilon, 
$$ 
then
$$
\mu_{(i,t)} (k) \leq  \left( r_{(i,t)}^{[k]} (1 | u )  +  r_{(i,t)}^{[k]} (0 | u ) \right) -  \left( r_{(i,t)}^{[k-1]} (1 | u )  +  r_{(i,t)}^{[k-1]} (0 | u ) \right) + \varepsilon .
$$
Here we can assume, without loss of generality that $L_t^i (u) = - \infty .$ Indeed if $L_t^i (u) > - \infty,$ let $u'$ be such that $L_t^i (u')= - \infty, u' \left( \left] - \infty ; L_t^i (u) \right] \right) =0 $ and $u' \left( \left] L_t^i (u) ; t \right] \right) =u \left( \left] L_t^i (u) ; t \right] \right) ,$ then $u'$ and $u$ are two equivalent configurations in terms of acceptance/rejection decision of the site $(i,t).$ Then
\begin{multline*}
\mu_{(i,t)} (k) - \varepsilon \\ \leq \inf_{ z\in D_{(i,t)}^k(u)   } \phi_i \left( \sum_j W_{ j \to i } \int_{ L_t^i (z) }^t g_j (t-s) dz_j(s) \right) - \inf_{ z\in D_{(i,t)}^{k-1}(u)   } \phi_i \left( \sum_j W_{ j \to i } \int_{ L_t^i (z) }^t g_j (t-s) dz_j(s) \right) \\ +  \sup_{ z\in D_{(i,t)}^{k-1}(u)   } \phi_i \left( \sum_j W_{ j \to i } \int_{ L_t^i (z) }^t g_j (t-s) dz_j(s) \right) - \sup_{ z\in D_{(i,t)}^k(u)   } \phi_i \left( \sum_j W_{ j \to i } \int_{ L_t^i (z) }^t g_j (t-s) dz_j(s) \right) 
\end{multline*}
Now we will simplify this expression detailing the configurations that realize the extrema.  In order to reach a lower-bound, we have to fix $z$ such that $L_t^i (z)=- \infty$ (we can do this since $L_t^i (u) = - \infty$) and whenever we have the choice for $z$  we also have to fix $z=1$ if the corresponding $W_{j \to i}$ is negative, else we have to fix $z=0.$  We do the opposite choice in order to reach an upper-bound.

Using condition (\ref{eq:Lip}), we obtain
\begin{multline*}
\mu_{(i,t)} (k)- \varepsilon \\ 
\leq  \gamma \left[ \left( \sum_{j \in V_i(k-1)} \left| W_{ j \to i } \right| \int_{ - \infty }^{t-k+1} g_j (t-s) dN_s^j  + \sum_{j \notin V_i(k-1)} \left| W_{ j \to i } \right| \int_{ - \infty }^t g_j (t-s) dN_s^j \right) \right. \\ 
- \left. \left( \sum_{j \in V_i(k)} \left| W_{ j \to i } \right| \int_{ - \infty }^{t-k} g_j (t-s) dN_s^j + \sum_{j \notin V_i(k)} \left| W_{ j \to i } \right| \int_{ - \infty }^t g_j (t-s) dN_s^j \right) \right] ,
\end{multline*}
so that finally 
$$
\mu_{(i,t)} (k) \leq \gamma \left( \sum_{j \in V_i(k-1)} \left| W_{ j \to i } \right| \int_{ t-k }^{t-k+1} g_j (t-s) dN_s^j + \sum_{j \in \partial V_i(k-1)} \left| W_{ j \to i } \right| \int_{t-k }^t g_j (\bar t-s) dN_s^j \right)
$$
which is the announced upper-bound (\ref{eq:ubmu}).
\end{proof}

Now, we will use the stationarity of the PRM ; this will allow us to omit the  dependence in time (by fixing $t=0$) since we are only interested in the expectation.  We have the following first upper-bound for  $M^{(i,t) }$ thanks to the inequality (\ref{eq:ubmu})
\begin{equation}\label{eq:condexp}
\frac{M^{(i,t) }}{\gamma}  \leq \sum_{k \geq 1} \left| C_{(i,t)} (k) \right| \left( \sum_{j \in V_i(k-1)} \left| W_{ j \to i } \right| \int_{ k-1 }^{k} g_j (s) dN_s^j  \right. \\
\left. + \sum_{j \in \partial V_i(k-1)} \left| W_{ j \to i } \right| \int_{ 0 }^{k} g_j (s) dN_s^j \right) .
\end{equation}

The difficulty to calculate the expectation of this upper-bound is that $N_s^j$ is present in each term of the product so that these terms are not independent. However, we have independence whenever the indexes denoting the neurons are different or when the intervals of time that we consider are disjoints.  We will therefore decompose the sums in the previous expression in order to isolate the products of non-independent terms. Then we calculate separately the expectations of these terms. For example, we have:
$$
\mathbb E \left[ \left( \int_0^k dN_s^j \right)  \left( \int_0^k  g_j(s) dN_s^j \right) \right] = \Lambda_j \left( k \Lambda_j +1 \right) \left( \int_0^k g_j(s)ds \right) .
$$

This result allows us to decompose  the above expectation as the sum of the covariance and the product of expectations. We can make such a decomposition for each product of non independent terms in (\ref{eq:condexp}).
Consequently, the expectation of the upper bound in (\ref{eq:condexp}) can be written as the sum of the covariances of the non-independent terms and the products of expectations of all the terms. After factorization, we finally obtain:
\begin{multline*}
\frac{\mathbb E \left( M^{(i,t)}\right)}{\gamma} 
 \leq \sum_{k \geq 1} \left[ \left( k \left( \sum_{j \in  V_i(k)} \Lambda_j \right)+1 \right) \left(   
\sum_{j \in V_i(k-1)} | W_{j \to i} | \Lambda_j \int_{k-1}^{k} g_j(s) ds \right. \right. \\
\left. \left.  +  \sum_{j \in \partial V_i(k-1)} | W_{j \to i} | \Lambda_j \int_0^{k} g_j(s) ds  \right)  \right].
\end{multline*}
Now we can use the assumption (\ref{eq:horrible}) of Theorem 1 in order to deduce that $ M := \sup_i \mathbb E\left( M^{(i,m)}\right) < 1 .$ Consequently, we have
$$\mathbb E (| C_n^{(i,t)} |) \le M^n \to 0 \mbox{ when } n \to \infty . $$
This ensures us that the process $\left( C_n^{(i,t)} \right)_{n \in \N}$ goes extinct a.s., or in other words that $N^{Stop} < \infty$ a.s. Consequently the perfect simulation algorithm ends in a finite time and this achieves the proof of the construction of the process.
\begin{flushright}
$\square$
\end{flushright}

\section*{Acknowledgments}
We thank A. Galves and P.A. Ferrari for many nice discussions on the subject.  P.H. thanks G. Lang very warmly for careful reading of the manuscript and valuable comments.

This research has been conducted as part of the project Labex MME-DII (ANR11-LBX-0023-01) and as part of the activities of FAPESP Research,
Dissemination and Innovation Center for Neuromathematics (grant
2013/07699-0, S.\ Paulo Research Foundation).

\end{document}